\documentclass[preprint,12pt]{elsarticle}

\usepackage{hyperref}
\usepackage{amsmath, amsthm, amssymb, amsfonts}
\allowdisplaybreaks
\usepackage{graphicx}

\newtheorem{thm}{Theorem}[section]
\newtheorem{lem}[thm]{Lemma}
\newtheorem{cor}[thm]{Corollary}
\newtheorem{prop}[thm]{Proposition}
\theoremstyle{definition}
\newtheorem{definition}[thm]{Definition}

\theoremstyle{remark}

\numberwithin{equation}{section}

\journal{Journal of Linear Algebra and its Applications}

\begin{document}
	
	\begin{frontmatter}
		
		\title{The equality of generalized matrix functions on the set of all symmetric matrices}
		
		
		\author[mymainaddress]{Ratsiri Sanguanwong}
		\ead{r.sanguanwong@gmail.com}
		
		\author[mysecondaryaddress]{Kijti Rodtes\corref{mycorrespondingauthor}}
		\cortext[mycorrespondingauthor]{Corresponding author}
		\ead{kijtir@nu.ac.th}
		
		\address[mymainaddress]{Department of Mathematics, Faculty of Science, Naresuan University, Phitsanulok 65000, Thailand}
		\address[mysecondaryaddress]{Department of Mathematics, Faculty of Science, Naresuan University, and Research Center for Academic Excellent in Mathematics, Phitsanulok 65000, Thailand}
		
		\begin{abstract}
			A generalized matrix function $d_\chi^G : M_n(\mathbb{C}) \rightarrow \mathbb{C}$ is a function constructed by a subgroup $G$ of $S_n$ and a complex valued function $\chi$ of $G$. 
				The main purpose of this paper is to find a necessary and sufficient condition for the equality of two generalized matrix functions on the set of all symmetric matrices, $\mathbb{S}_n(\mathbb{C})$. 
			In order to fulfill the purpose, a symmetric matrix $S_\sigma$ is constructed and $d_\chi^G(S_\sigma)$ is evaluated for each $\sigma \in S_n$. 
				By applying the value of $d_\chi^G(S_\sigma)$, it is shown that $d_\chi^G(AB) = d_\chi^G(A)d_\chi^G(B)$ for each $A, B \in \mathbb{S}_n(\mathbb{C})$ if and only if $d_\chi^G = \det$.
			Furthermore, a criterion when $d_\chi^G(AB) = d_\chi^G(BA)$ for every $A, B \in \mathbb{S}_n(\mathbb{C})$, is established. 
		\end{abstract}
		
		
	\end{frontmatter}
	
	
	\section{Introduction}
	In the theory of matrix, a determinant and a permanent are two well-known functions having many useful applications. 
		A concept of a generalized matrix function is to generalize notions of permanents and determinants. 
	Jafari and Madadi showed that two generalized matrix functions $d_\varphi^H$ and $d_\psi^K$ are equal on the set of all nonsingular matrices if and only if the extended function $\hat{\varphi}$ of $\varphi$ and $\hat{\psi}$ of $\psi$ are equal. 
		In fact, by applying the proof of their results, the condition also holds for $M_n(\mathbb{C})$, \cite{2014jm2}.
	
	In 1986, Bosch discovered the factorization of a square matrix in term of the product of two symmetric matrices, \cite{1986b}, which causes our attraction to characterize the equality of generalized matrix functions on the set of all symmetric matrices. 
		If $\varphi$ is not a real valued function and $\psi$ conjugates to $\varphi$, we can calculate that $d_\varphi^H$ and $d_\psi^K$ is equal on the set of all symmetric matrices. 
	We construct a symmetric matrix $S_\sigma$, for each $\sigma \in S_n$ and define an equivalence relation on $S_n$. By using such matrices along with the relation, a necessary and sufficient condition for the equality is proven (see Theorem \ref{thm1}). 
		By applying $S_\sigma$ for some $\sigma \in S_n$, we obtain Theorem \ref{thm15} which related to Corollary 2.4. in \cite{2014jm1}. 
	By these results, to obtain the equalities on the set of all symmetric matrices, it suffices to work on the permutation group which is finite.

	\section{Preliminaries}
	Let $G$ be a subgroup of $S_n$. 
		A \textit{character} of $G$ is a complex valued function of $G$ defined by $g \mapsto tr(\varphi_g)$, where $\varphi$ is a homomorphism from $G$ into $GL_m(\mathbb{C})$; $m$ is a positive integer and $\varphi_g := \varphi(g)$. 
	A character is a class(conjugacy class) function. 
		Many authors use a character to define a generalized matrix function, but, in this paper, we require such function to be a complex valued function because our main result hold whether the function is a character or not. 
	However, we also investigate some interesting consequences when using a character.
	
	Let $G \leq S_n$ and $\chi : G \rightarrow \mathbb{C}$ be a complex valued function. 
		For each $A \in M_n(\mathbb{C})$, denote $[A]_{ij}$ the entry in the $i$-th row and $j$-th column of $A$. A function $d_\chi^G : M_n(\mathbb{C}) \rightarrow \mathbb{C}$ defined by
			\begin{equation*}
			d_\chi^G(A) = \sum\limits_{\sigma \in G} \chi(\sigma) \prod\limits_{i=1}^{n} [A]_{i\, \sigma(i)}
			\end{equation*}
		is called a\textit{ generalized matrix function} associated with $G$ and $\chi$.
	We can see that $d_\varepsilon^{S_n}$ and $d_1^{S_n}$ are a determinant function and a permanent function, where $\varepsilon$ and $1$ are the principal character and the alternating character, respectively. 
		Moreover, for each $A \in M_n(\mathbb{C})$,
			\begin{equation}\label{eq08}
			d_\chi^G(A) = \sum\limits_{\sigma \in S_n} \hat{\chi}(\sigma) \prod\limits_{i=1}^{n} [A]_{i\, \sigma(i)},
			\end{equation}
		where $\hat{\chi}$ is an extension of $\chi$ which vanishes outside $G$. 
	In this paper, we always consider $d_\chi^G$ as in the form (\ref{eq08}).
	
	A complex square matrix $A$ is called a \textit{symmetric matrix} if $A^T = A$.
		We denote the set of all $n \times n$ complex symmetric matrices by $\mathbb{S}_n(\mathbb{C})$. 
	For each $A, B \in \mathbb{S}_n(\mathbb{C})$, we have $(AB)^T = BA$.

	\section{Main results}
	We know that every permutation in $S_n$ can be uniquely written as a product of some disjoint cycles. 
		Through out this paper, for every permutation $\sigma$ in $S_n$, the decomposition of $\sigma$ means the expression of $\sigma$ in term of a product of disjoint cycles not including a cycle of length 1. 
	For arbitrary positive integer $m$, the set $\{1, \dots, m\}$ is symbolized by $[m]$. 
		For each $\sigma \in S_n$, we denote $\operatorname{Fix}(\sigma) := \{i \in [n] \mid \sigma(i) = i\}$. 
	For every cycle $\omega = (a_1 \, \dots \, a_s)$, we can see directly that $\{a_1, \dots, a_s\} = \operatorname{Fix}(\omega)^c$, the complement of $\operatorname{Fix}(\omega)$. 
		It is also clear that 
			\begin{equation*}
			\hbox{\textquotedblleft$\omega$ and $\pi$ are disjoint cycles if and only if $\operatorname{Fix}(\omega)^c \cap \operatorname{Fix}(\pi)^c = \emptyset$\textquotedblright.}
			\end{equation*}
	Furthermore, 
			\begin{equation*}
			\hbox{\textquotedblleft$\operatorname{Fix}(\omega\pi)^c = \operatorname{Fix}(\omega)^c \cup \operatorname{Fix}(\pi)^c$ if $\omega$ and $\pi$ are disjoint cycles\textquotedblright.} 
			\end{equation*}
		By using these facts, the following lemma holds.
		
	\begin{lem}\label{lem0} 
	Let $\sigma \in S_n$. 
		Suppose that $\sigma = \sigma_1 \cdots \sigma_k$ is the decomposition of $\sigma$ and $i \in [n]$. 
	Then the following statements hold.
			\begin{enumerate}
			\item[(1)] $i \in \operatorname{Fix}(\sigma)$ if and only if $i \in \operatorname{Fix}(\sigma_j)$ for every $j \in [k]$. 
			\item[(2)] If $i \in \operatorname{Fix}(\sigma)^c$, then there exists uniquely element $j \in [k]$ such that $i \in \operatorname{Fix}(\sigma_j)^c$.
			\item[(3)] $\operatorname{Fix}(\sigma)^c = \operatorname{Fix}(\sigma_1)^c \cup \dots \cup \operatorname{Fix}(\sigma_k)^c$.
			\end{enumerate}
	\end{lem}

	If $i \in \operatorname{Fix}(\sigma_j)^c$, then, by (1), $i \in \operatorname{Fix}(\sigma)^c$. 
		Moreover, because of (2), $i \in \operatorname{Fix}(\sigma_l)$ if $l \neq j$. 
	So, $\sigma(i) = \sigma_j(i)$ if $i \in \operatorname{Fix}(\sigma_j)^c$. 
		Note that $i \in \operatorname{Fix}(\sigma_j)^c$ implies $\sigma_j^t(i) \in \operatorname{Fix}(\sigma_j)^c$ for every natural number $t$. 
	Thus, by using mathematical induction,
			\begin{equation*}
			\hbox{\textquotedblleft$\sigma^t(i) = \sigma_j^t(i)$ if $i \in \operatorname{Fix}(\sigma_j)^c$\,\textquotedblright.}
			\end{equation*}

	Now, for the decomposition $\sigma = \sigma_1\sigma_2\cdots\sigma_k$, we define the set $[\sigma]$ as follow.
			\begin{align*}
			[\sigma]:=\{\sigma_1^{n_1}\sigma_2^{n_2}\cdots\sigma_k^{n_k} \mid n_1,\dots,n_k \in \{1, -1\}\}.
			\end{align*}
	
	Define a relation $\sim$ on $S_n$ by, for each $\sigma, \tau \in S_n$, $\tau \sim \sigma$ if and only if $\tau \in [\sigma]$. 
		We can see that, by the definition of $[\sigma]$, $\sim$ is an equivalence relation and $[\sigma]$ is the equivalence class of $\sigma$. 
	
	Let $\omega = (a_1 \, a_2\dots\, a_s)$ be a cycle in $S_n$. 
		If $s$ is even, we construct
			\begin{align*} 
			\omega_{(a_1)} := (a_1 \, a_2)(a_3 \, a_4)\cdots(a_{s-1} \, a_s) \text{\ \ \ and\ \ \ } \omega_{(a_s)} := (a_s \, a_1)(a_2 \, a_3)\cdots(a_{s-2} \, a_{s-1}),
			\end{align*} 
		and then denote $S(\omega) := \{\omega_{(a_1)},\omega_{(a_s)}\}$. 
	Define the set $X_\omega$ by 
			\begin{equation*}
			X_\omega := \left\{
			\begin{array}{ll}
			\{\omega,\omega^{-1}\}, & \hbox{if $\omega$ is a cycle of odd length,} \\
			\{\omega,\omega^{-1}\} \cup S(\omega), & \hbox{if $\omega$ is a cycle of even length.}
			\end{array}
			\right.
			\end{equation*}
		Note that, if $\omega$ is a transposition, then $X_\omega = \{\omega\}$.
	For every two subsets $A, B$ of $S_n$, denote $AB = \{\sigma\tau \mid \sigma \in A \text{ and } \tau \in B\}$. 
		For any $\sigma \in S_n$, we define 
			\begin{equation*}
			X_\sigma := X_{\sigma_1}\cdots X_{\sigma_k},
			\end{equation*}
		where $\sigma=\sigma_1\cdots\sigma_k$ is the decomposition of $\sigma$. 
	By the definition of $[\sigma]$, we obtain that $[\sigma] \subseteq X_\sigma$.
		So, we can classify elements in $S_n$ into two types.
		
	\begin{definition}
		The permutation $\sigma$ is said to be \textit{type-I} if $X_\sigma = [\sigma]$ and it is said to be \textit{type-II} if it is not a type-I.
	\end{definition}

	Then every cycle of odd length is type-I. 
		If $\omega$ is a cycle of even length, it is clear by the definition of $[\omega]$ and $X_\omega$ that $\omega$ is type-I if and only if it is a transposition. 
	By the direct computation, $S(\omega) = S(\omega^{-1})$, for each cycle $\omega$ of even length, and hence, $X_\pi = X_{\pi^{-1}}$, for every cycle $\pi$ in $S_n$. 
		This implies that $X_\sigma = X_{\tau}$, for each $\tau \in [\sigma]$.
	
	\begin{prop}\label{prop1} Let $\sigma \in S_n$. Then $\sigma$ is type-II if and only if the decomposition of $\sigma$ contains a type-II cycle.
	\end{prop}
	\begin{proof}
		Let $\sigma = \sigma_1\cdots\sigma_k$ be the decomposition of $\sigma$.
			Assume that the decomposition does not contain any type-II cycle. 
		Then $[\sigma_i] = \{\sigma_i,\sigma_i^{-1}\} = X_{\sigma_i}$ for every $1 \leq i \leq k$.
			So, by the definition of $[\sigma]$,
				\begin{equation*}
				X_\sigma = X_{\sigma_1}\cdots X_{\sigma_k} = [\sigma_1] \cdots [\sigma_k] = [\sigma],
				\end{equation*}
			that is, $\sigma$ is type-I. 
		Conversely, suppose that the decomposition contains a type-II cycle, says $\sigma_1$.
			By considering $\sigma_1$ as $\sigma_1 = (a_1 \, \dots \, a_s)$, we have $s \geq 4$ and $(\sigma_1)_{(a_1)}\sigma_2\sigma_3\cdots\sigma_k \in X_\sigma$. 
		Since $(\sigma_1)_{(a_1)}$ is not equal to $\sigma_1$ nor $\sigma_1^{-1}$ and because of the uniqueness of the decomposition of a permutation, $(\sigma_1)_{(a_1)}\sigma_2\sigma_3\cdots\sigma_k \not\in [\sigma]$, that is, $X_{\sigma} \neq [\sigma]$. Thus $\sigma$ is type-II.
	\end{proof}
	
	\begin{prop}\label{prop3}\label{prop2} 
	Let $\sigma \in S_n$ and $\tau \in X_\sigma$. 
		Then the following statements hold.
			\begin{enumerate}
			\item[(1)] $\operatorname{Fix}(\sigma) = \operatorname{Fix}(\tau)$.
			\item[(2)] $X_\tau \subseteq X_\sigma$.
			\item[(3)] $X_\tau = X_\sigma$ if and only if $\tau \in [\sigma]$.
			\end{enumerate}
	\end{prop}
	\begin{proof}
	(1) It is straightforward by the definition of $X_\sigma$ that $\operatorname{Fix}(\sigma)^c = \operatorname{Fix}(\tau)^c$. 
		Hence $\operatorname{Fix}(\sigma) = \operatorname{Fix}(\tau)$.
		
	(2) Let $\sigma_1 \cdots \sigma_k$ be the decomposition of $\sigma$. 
		Then $\tau = \tau_1 \cdots \tau_k$, where $\tau_i \in X_{\sigma_i}$. 
	If $\tau_i = \sigma_i$ or $\sigma_i^{-1}$, then $X_{\tau_i} = X_{\sigma_i}$.
		Suppose that $\sigma_i$ is type-II and $\tau_i$ is not $\sigma_i$ nor $\sigma_i^{-1}$. 
	We may consider $\sigma_i$ in the form $\sigma_i = (a_1 \, \dots \, a_s)$. 
		Then $\tau_i$ is either $(\sigma_i)_{(a_1)}$ or $(\sigma_i)_{(a_s)}$. 
	Both cases imply that $X_{\tau_i} = \{\tau_i\} \subseteq X_{\sigma_i}$. 
		Thus $X_\tau = X_{\tau_1} \cdots X_{\tau_k} \subseteq X_{\sigma_1} \cdots X_{\sigma_k} = X_\sigma$. 
		
	(3) If $\tau \in [\sigma]$, we can conclude that $X_\sigma = X_\tau$ as we discuss before announcing Proposition \ref{prop1}. 
		On the other hand, suppose that $X_\tau = X_\sigma$ and $\tau \not\in [\sigma]$. 
	By the above argument, there exists $i \in [k]$ such that $\tau_i = (\sigma_i)_{(a_1)}$ or $\tau_i = (\sigma_i)_{(a_s)}$, which implies that $X_{\tau_i} \neq X_{\sigma_i}$. 
		This means that $X_\tau \neq X_\sigma$ which is a contradiction. Therefore $\tau \in [\sigma]$.
	\end{proof}
	
	\begin{prop}\label{prop4} 
	Let $\sigma \in S_n$. 
		Then the following statements hold.
			\begin{enumerate}
			\item[(1)] $X_\sigma$ is a union of the equivalence classes of permutations in $S_n$.
			\item[(2)] If $\sigma$ contains exactly $l$ cycles of type-II, then the decomposition of each element of $X_\sigma$ contains at most $l$ type-II cycles.
			\item[(3)] If $\sigma$ contains exactly $l$ cycles of type-II, then $[\sigma]$ is the set of all elements of $X_\sigma$ containing exactly $l$ type-II cycles.  
			\end{enumerate}	
	\end{prop}
	\begin{proof}
	If $\sigma$ is type-I, then $X_\sigma = [\sigma]$. 
		Suppose that $\sigma$ is type-II. 
	Then the decomposition of $\sigma$ contains a type-II cycle. 
		Let $\sigma = \sigma_1 \cdots \sigma_l \sigma_{l+1} \cdots \sigma_k$ be the decomposition of $\sigma$, where $\sigma_1, \dots ,\sigma_l$ are all type-II cycles contained in this decomposition. 
	For each $j \in [l]$, define 
			\begin{equation*}
			Seq(j) := \{(x_1, \dots ,x_j) \in \mathbb{N}^j  \mid x_i\leq l \text{ and } x_1 < x_2 < \cdots < x_j\}.
			\end{equation*}
		For each sequence $(S) = (x_1, \dots, x_j) \in Seq(j)$, denote $S:= \{x_1, \dots ,x_j\}$ and
			\begin{equation*}
			X_\sigma^S := \prod\limits_{i \in S}S(\sigma_i)\prod\limits_{i\not\in S}[\sigma_i].
			\end{equation*}
	Straightforward by this definition, we have $X_\sigma^S \subset X_\sigma$. 
		Note that, for each $1\leq j \leq l$,
			\begin{equation*}
			S(\sigma_j) = \{(\sigma_j)_{(a_1^j)}, (\sigma_j)_{(a_{s_j}^j)}\} = [(\sigma_j)_{(a_1^j)}] \cup [(\sigma_j)_{(a_{s_j}^j)}],
			\end{equation*}
		where $\sigma_j = (a_1^j \, \dots \, a_{s_j}^j)$. 
	Let $\mathbb{I} = \{a_1^{x_1}, a_{s_{x_1}}^{x_1}\} \times \cdots \times \{a_1^{x_j}, a_{s_{x_j}}^{x_j}\}$. 
		Then
			\begin{equation}\label{eq10}
			X_\sigma^S = \bigcup\limits_{(S_{x_1}, \dots, S_{x_j}) \in \mathbb{I}}(\prod\limits_{i \in S}[(\sigma_i)_{(S_i)}]\prod\limits_{i \not\in S}[\sigma_i]) = \bigcup\limits_{(S_{x_1}, \dots, S_{x_j}) \in \mathbb{I}}[\prod\limits_{i \in S}(\sigma_i)_{(S_i)}\prod\limits_{i \not\in S}\sigma_i].
			\end{equation} 
	Thus $X_\sigma^S$ is a union of some equivalence classes of $S_n$.
		So, it suffices to show that $X_\sigma = [\sigma] \cup (\bigcup\limits_{j = 1}^l \bigcup\limits_{(S) \in Seq(j)} X_\sigma^S)$. 
	It is clear that
			\begin{equation*}
			[\sigma] \cup (\bigcup\limits_{j = 1}^l \bigcup\limits_{(S) \in Seq(j)} X_\sigma^S) \subseteq X_\sigma.
			\end{equation*} 
		Let $\tau = \tau_1 \cdots \tau_k \in X_\sigma$, where $\tau_i \in X_{\sigma_i}$. 
	Suppose that $\tau \not \in [\sigma]$. 
		By using the same argument as in the proof of Proposition \ref{prop3}, the decomposition of $\tau$ contains an element in $S(\sigma_i)$ for some $1 \leq i \leq l$. 
	Without loss of generality, we may assume that the decomposition of $\tau$ contains an element in each $S(\sigma_1), \dots, S(\sigma_j)$ but not contain element in $S(\sigma_{j+1}), S(\sigma_{j+2}), \dots, S(\sigma_l)$. 
		Then 
			\begin{equation*}
			\tau \in \prod\limits_{i = 1}^jS(\sigma_i) \prod\limits_{i = j+1}^k [\sigma_i] = X_\sigma^{S_0},
			\end{equation*}
		where $(S_0) = (1, 2, \dots, j)$. 
	This implies that each element in $X_\sigma$ is contained in $[\sigma]$ or $(\bigcup\limits_{j = 1}^l \bigcup\limits_{(S) \in Seq(j)} X_\sigma^S)$. 
		So, 
			\begin{equation}\label{eq01}
			X_\sigma = [\sigma] \cup (\bigcup\limits_{j = 1}^l \bigcup\limits_{(S) \in Seq(j)} X_\sigma^S).
			\end{equation}
	This proves that (1) holds. 
		For each $i \in [l]$, each cycle in $[\sigma_i]$ is type-II while every element in $S(\sigma_i)$ is type-I. 
	Since $\sigma_i$ is type-I if $i \geq l$, by considering the construction of $X_\sigma^S$, where $S \in Seq(j)$, we can see that the decomposition of each element in $X_\sigma^S$ contains exactly $l-j$ type-II cycles. 
		Thus (2) and (3) hold.
	\end{proof}
	
    \begin{definition}
	For each $\sigma \in S_n$, we define the $n \times n$ matrix $S_\sigma$ as
			\begin{equation*}
			[S_\sigma]_{ij} = \left\{
			\begin{array}{ll}
			1, & \hbox{if $\sigma(i)=j$ or $\sigma^{-1}(i)=j$,} \\
			0, & \hbox{otherwise.}
			\end{array}
			\right.
			\end{equation*}
    \end{definition}
	
	It is not difficult to see that $S_\sigma$ is a symmetric matrix and $S_\sigma = S_{\sigma^{-1}}$. 
		Furthermore, $S_\sigma=P_\sigma$ if and only if $\sigma^2 = id$, where $P_\sigma$ is the permutation matrix corresponding to $\sigma$.
	
	\begin{thm}\label{lem1}
	Let $G \leq S_n$ and $\chi:G \rightarrow \mathbb{C}$ be a complex valued function. 
		For each $\sigma \in S_n$, $d_\chi^G(S_\sigma)=\sum\limits_{\tau\in X_\sigma} \hat{\chi}(\tau)$.
	\end{thm}
	\begin{proof}
	Let $\sigma \in S_n$. 
		To prove that $d_\chi^G(S_\sigma)=\sum\limits_{\tau\in X_\sigma}\hat{\chi}(\tau)$, it suffices to show that $[S_\sigma]_{i \,\tau(i)} = 1$ for each $i \in [n]$ if and only if $\tau \in X_\sigma$.
	Suppose that $\tau \in X_\sigma$. 
		To show that $[S_\sigma]_{i \, \tau(i)} = 1$ for each $i \in [n]$, it is enough to show that, for every $i \in [n]$, $\tau(i) = \sigma(i)$ or $\tau(i) = \sigma^{-1}(i)$. Let $i \in [n]$. 
	Note that, by Proposition \ref{prop2}, $\tau \in X_\sigma$ implies $\operatorname{Fix}(\sigma) = \operatorname{Fix}(\tau)$. 
		Thus, if $i \in \operatorname{Fix}(\sigma)$, then $\tau(i) = i = \sigma(i)$. 
	Now, suppose that $i \in \operatorname{Fix}(\sigma)^c$. 
		Then there exists a unique cycle in the decomposition of $\sigma$, says $\pi$, such that $\sigma(i) = \pi(i)$. 
	Since $\tau \in X_\sigma$, there exists a permutation $\theta \in X_\pi$ such that $\theta$ is contained in the decomposition of $\tau$. 
		If $\pi$ is a cycle of odd length, then $X_\pi = \{\pi, \pi^{-1}\}$, which implies that $\theta = \pi$ or $\pi^{-1}$. 
	Since $i \in \operatorname{Fix}(\pi)^c$ and $\theta \in X_\pi$, $i \in \operatorname{Fix}(\theta)^c$. 
		By Lemma \ref{lem0}(2), $\tau(i) = \theta(i)$ which is $\pi(i)$ or $\pi^{-1}(i)$. 
	In other words, 
			\begin{equation*}
			\hbox{$\tau(i) = \sigma(i)$ or $\tau(i) = \sigma^{-1}(i)$.}
			\end{equation*}
		If $\pi$ is a cycle of even length, then $X_\pi = \{\pi, \pi^{-1}\} \cup S(\pi)$. 
	By using the same argument as the above, the proof is done in case $\theta \in \{\pi, \pi^{-1}\}$. 
		Suppose that $\theta \in S(\pi)$. 
	Since $i \in \operatorname{Fix}(\pi)^c$, we can write $\pi$ in the form $(i \,  \pi(i) \, \pi^2(i) \, \dots \, \pi^{s-1}(i))$, where $s$ is the length of $\pi$. 
		Then $\theta = \pi_{(i)}$ or $\theta = \pi_{(\pi^{s-1}(i))}$. Since $\tau(i) = \theta(i)$, 
			\begin{equation*}
			\hbox{$\tau(i) = \pi(i) = \sigma(i)$ or $\tau(i) = \pi^{s-1}(i) = \pi^{-1}(i) = \sigma^{-1}(i)$.}
			\end{equation*}
	Thus $[S_\sigma]_{i \, \tau(i)} = 1$ for each $i \in [n]$.
		
	Conversely, suppose that $[S_\sigma]_{i \, \tau(i)} = 1$ for every $i \in [n]$. 
		Thus, for each $i \in [n]$, 
			\begin{equation}\label{eq09}
			\hbox{$\tau(i) = \sigma(i)$ or $\tau(i) = \sigma^{-1}(i)$,}
			\end{equation}
		which implies that $i \in \operatorname{Fix}(\tau)$ if and only if $i \in \operatorname{Fix}(\sigma)$. 
	So we obtain that 
			\begin{equation} \label{eq02}
			\operatorname{Fix}(\sigma) = \operatorname{Fix}(\tau).
			\end{equation}
		Let $\sigma=\sigma_1\cdots \sigma_k$ and $\tau = \tau_1\cdots \tau_l$ be the decompositions of $\sigma$ and $\tau$, respectively. 
	First of all, we will show that, for each $j \in [k]$, there exists $\theta \in X_{\sigma_j}$ such that $\theta$ is a factor of $\tau$. 
		It suffices to work with $\sigma_1$ because the product of disjoint cycles are commute. 
	Let $a \in \operatorname{Fix}(\sigma_1)^c$. 
		Clearly, $a \not\in \operatorname{Fix}(\sigma) = \operatorname{Fix}(\tau)$.
	There exists a unique cycle in the decomposition of $\tau$, without loss of generality, we may assume that it is $\tau_1$, such that $\tau_1(a) = \tau(a)$.
		If $\sigma_1$ is a transposition, then
			\begin{equation*}
			\sigma(a) = \sigma_1(a) = \sigma_1^{-1}(a) = \sigma^{-1}(a).
			\end{equation*}
	By applying (\ref{eq09}), this condition implies that 
			\begin{equation*}
			\hbox{$\tau_1(a) = \tau(a) = \sigma(a) = \sigma_1(a)$ and $\tau_1^2(a) = \tau^2(a) = \tau\sigma(a)$}.
			\end{equation*}
		By (\ref{eq09}) again, $\tau\sigma(a) = \sigma^{-1}\sigma(a)$ or $\tau\sigma(a) = \sigma^2(a)$. Since $\sigma^{-1}\sigma(a) = a = \sigma_1^2(a) = \sigma^2(a)$, $\tau_1^2(a) = a$. 
	Thus $\tau_1 = \sigma_1$ if $\sigma_1$ is a transposition. 
		This case is done by choosing $\theta = \tau_1$.
		
	Assume that $\sigma_1$ is not a transposition. 
		Without loss of generality, we can write 
			\begin{equation*}
			\hbox{$\sigma_1 = (a \, \sigma_1(a) \, \sigma_1^2(a) \, \dots \, \sigma_1^{s-1}(a))$ and $\tau_1 = (a \, \tau_1(a) \, \tau_1^2(a) \, \dots \, \tau_1^{t-1}(a))$.}
			\end{equation*}
	Claim that, for each $x \in [t]$, $\tau_1^{x}(a) \in \operatorname{Fix}(\sigma_1)^c$.
		To verify the claim, suppose that $x$ is the smallest element in $[t-1]$ such that $\tau_1^{x}(a) \in \operatorname{Fix}(\sigma_1)$. 
	Since $\tau_1^x(a) = \tau_1(\tau_1^{x-1}(a)) = \tau(\tau_1^{x-1}(a))$, by applying (\ref{eq09}),
			\begin{equation*}
			\hbox{$\sigma_1(\tau_1^{x-1}(a)) = \sigma(\tau_1^{x-1}(a)) = \tau(\tau_1^{x-1}(a)) = \tau_1^x(a)$}
			\end{equation*} 
		or
			\begin{equation*}
			\hbox{$\sigma^{-1}_1(\tau_1^{x-1}(a)) = \sigma^{-1}(\tau_1^{x-1}(a)) = \tau(\tau_1^{x-1}(a)) = \tau_1^x(a)$}.
			\end{equation*} 
	If $\sigma_1(\tau_1^{x-1}(a)) = \tau_1^x(a)$, then 
			\begin{equation*}
			\tau_1^{x-1}(a) = \sigma_1^{-1}\sigma_1\tau_1^{x-1}(a) = \sigma_1^{-1}\tau_1^x(a) = \tau_1^{x}(a),
			\end{equation*}
	which is a contradiction. 
		By the similar reasoning, the contradiction also occurs when $\sigma_1^{-1}(\tau_1^{x-1}(a)) = \tau_1(\tau_1^{x-1}(a)) = \tau_1^x(a)$ as we can see that 
			\begin{equation*}
			\tau_1^{x-1}(a) = \sigma_1\sigma_1^{-1}\tau_1^{x-1}(a) = \sigma_1\tau_1^x(a) = \tau_1^{x}(a). 
			\end{equation*}
	Thus, the claim is true. 
		By applying the claim, we also have that $\operatorname{Fix}(\tau_1)^c \subseteq \operatorname{Fix}(\sigma_1)^c$, so $s \geq t$. 
	We now consider $\tau_1$ as four possible cases.
		
	Case 1: $\tau_1$ is not a transposition and $\tau(a) = \sigma(a)$. 
		Then $\tau_1(a) = \tau(a) = \sigma(a) = \sigma_1(a)$. 
	This indicates that $\tau_1^2(a) = \tau_1\sigma_1(a)$. 
		If $\tau_1\sigma_1(a) = \tau\sigma_1(a) = \sigma^{-1}\sigma_1(a) = \sigma_1^{-1}\sigma_1(a) = a$, then $\tau_1^2(a) = a$, which is not possible because $\tau_1$ is not a transposition. 
	Then $\tau_1^2(a) = \sigma\sigma_1(a) = \sigma_1^2(a)$. 
		By using the above arguments, we can show that 
			\begin{equation*}
			\hbox{ $\tau_1^m(a)  = \sigma_1^m(a)$, for each $m \in [t-1]$.}
			\end{equation*}
	Moreover, $\sigma_1^t(a) = \sigma\sigma_1^{t-1}(a) = \sigma\tau_1^{t-1}(a)$. 
		If $\sigma\tau_1^{t-1}(a) = \tau^{-1}\tau_1^{t-1}(a)$, then $\sigma_1^t(a) = \tau_1^{t-2}(a) = \sigma_1^{t-2}(a)$, that is, $\sigma_1^2(a) = a$, which is a contradiction because $\sigma_1$ is not a transposition. 
	Then
			\begin{equation*}
			\sigma_1^t(a) = \sigma\tau_1^{t-1}(a) = \tau\tau_1^{t-1}(a) = \tau_1^t(a) = a.
			\end{equation*}
		So, we obtain that $s = t$ and hence $\tau_1 = \sigma_1$. 
	By setting $\theta = \tau_1$, this case is done.
		
	Case 2: $\tau_1$ is not a transposition and $\tau(a) = \sigma^{-1}(a)$. 
		By similar argument as Case 1, we have $\tau_1 = \sigma_1^{-1}$. 
	Thus this case is proven by putting $\theta = \tau_1$.
		
	Case 3: $\tau_1$ is a transposition and $\tau(a) = \sigma(a)$. 
		Then $\tau_1(a) = \sigma_1(a)$ and $\tau_1^2(a) = a$. 
	This implies that $\operatorname{Fix}(\tau_1)^c = \{a, \sigma_1(a)\}$. 
		Since $\operatorname{Fix}(\tau) = \operatorname{Fix}(\sigma)$ and $\sigma_1^2(a) \in \operatorname{Fix}(\sigma)^c = \operatorname{Fix}(\tau)^c$, there exists a unique cycle in the decomposition of $\tau$, says $\tau_2$ such that $\sigma_1^2(a) \in \operatorname{Fix}(\tau_2)^c$. 
	Because $\sigma_1$ is not a transposition, we have $\sigma_1^2(a) \neq a$.
		Since $\sigma_1^2(a)\not\in \operatorname{Fix}(\tau_1)^c$, we can conclude that $\tau_1 \neq \tau_2$, that is, $\tau_1$ and $\tau_2$ are disjoint. 
	If $\tau\sigma_1^2(a) = \sigma^{-1}\sigma_1^2(a)$, then $\tau_2\sigma_1^2(a) = \sigma_1(a) = \tau_1(a)$, which is a contradiction because $\tau_1$ and $\tau_2$ are disjoint. 
		This implies that $\tau_2\sigma_1^2(a) = \sigma\sigma_1^2(a) = \sigma_1^3(a)$. 
	If $\tau_2$ is not a transposition, by using the same argument as case 1, $\tau_2 = \sigma_1$. 
		This also contradicts to the fact that $\tau_1$ and $\tau_2$ are disjoint. 
	Thus $\tau_2$ is a transposition. 
		Note that $\sigma_1$ is not a transposition, that is, $s \geq 3$. If $s = 3$, then $\tau_1 = (a \, \sigma_1(a))$ and $\tau_2 = (\sigma_1^2(a) \, a)$, which is a contradiction, and hence $s \geq 4$. 
	If $s = 4$, this case is proved by choosing $\theta = \tau_1\tau_2 = (\sigma_1)_{(a)}$. 
		Suppose that $s > 4$, recall the fact that $\operatorname{Fix}(\sigma)^c = \operatorname{Fix}(\tau)^c$ again, there exists a unique cycle in the decomposition of $\tau$, says $\tau_3$ such that $\sigma_1^4(a) \in \operatorname{Fix}(\tau_3)^c$. 
	By using the same argument as above, we obtain the conclusion that $\tau_3$ is a transposition and $\tau_3\sigma_1^4(a) = \sigma_1^5(a)$. 
		By repeating this method, we can conclude that, for each $m \in \{2, \dots, \lfloor\frac{s}{2}\rfloor\}$,
			\begin{equation*}
			\tau_m = (\sigma_1^{2m-2}(a) \, \sigma_1^{2m-1}(a)).
			\end{equation*}
	If $s$ is odd, then
			\begin{equation*}
			\tau_{\lfloor\frac{s}{2}\rfloor} = (\sigma_1^{s-1}(a) \, \sigma_1^s(a)) = (\sigma_1^{s-1}(a) \, a),
			\end{equation*}
	which is a contradiction because $\tau_1$ and $\tau_{\lfloor\frac{s}{2}\rfloor}$ are disjoint. 
		Thus $s$ is even. 
	Moreover, $\tau_1 \tau_2 \cdots \tau_{\frac{s}{2}} = (\sigma_1)_{(a)}$. 
		Choose $\theta = (\sigma_1)_{(a)}$. 
	This case is verify.
		
	Case 4: $\tau_1$ is a transposition and $\tau(a) = \sigma^{-1}(a)$. 
		By using the similar reasoning as Case 3 and assigning $\theta = (\sigma_1)_{(\sigma_1^{s-1}(a))}$, we obtain that $\theta$ is contained in the decomposition of $\tau$, and hence this case is done.
		
	Now, we have the fact that, for every cycle $\pi$ in the decomposition of $\sigma$, there exists exactly one element $\theta \in X_\pi$ that contained in the decomposition of $\tau$. 
		So, $\tau$ can be expressed as $\tau = \theta_1 \cdots \theta_q$, where $q \geq k$, $\theta_i \in X_{\sigma_i}$, for each $i \in [k]$, and $\theta_i$ and $\theta_j$ are disjoint if $i \neq j$. If $q \neq k$, then
			\begin{equation*} 
			\operatorname{Fix}(\tau)^c = \bigcup\limits_{i = 1}^q \operatorname{Fix}(\theta_i)^c \neq \bigcup\limits_{i = 1}^k \operatorname{Fix}(\theta_i)^c = \bigcup\limits_{i = 1}^k \operatorname{Fix}(\sigma_i)^c = \operatorname{Fix}(\sigma)^c,
			\end{equation*}
		which contradicts to (\ref{eq02}). 
	This implies that $q = k$ and hence $\tau \in X_{\sigma}$. 
		Thus the statement of this theorem is proven.
	\end{proof} 
	
	\begin{cor} 
	Let $\sigma \in S_n$. 
		Then $\operatorname{perm}(S_\sigma) = |X_\sigma|$, where $\operatorname{perm}(S_\sigma)$ is the permanent of $S_\sigma$.
	\end{cor}
	
	By using Theorem \ref{lem1}, we can find a necessary and sufficient condition for the equality of two generalized matrix functions on the set of all symmetric matrices as below.
	
	\begin{thm}\label{thm1} 
	Let $H$ and $K$ be subgroups of $S_n$ and $\phi$ and $\psi$ complex valued functions of $H$ and $K$, respectively. 
		Then the following conditions are equivalent.
			\begin{enumerate}
			\item[(1)] $d_\phi^H(A) = d_\psi^K(A)$ for each $A \in \mathbb{S}_n(\mathbb{C})$,
			\item[(2)] $d_\phi^H(S_\sigma) = d_\psi^K(S_\sigma)$ for every $\sigma \in S_n$,
			\item[(3)] $\sum\limits_{\tau\in[\sigma]}\hat{\phi}(\tau)=\sum\limits_{\tau\in[\sigma]}\hat{\psi}(\tau)$ for every $\sigma\in S_n$.
			\end{enumerate}
	\end{thm}
	\begin{proof}
	It is obvious that (1) implies (2). 
		Suppose that (2) is true. 
	Let $\sigma = \sigma_1\cdots \sigma_k$ be a decomposition of $\sigma$. 
		If $\sigma$ is type-I, i.e., $X_\sigma = [\sigma]$, then, by Theorem \ref{lem1},
			\begin{equation*}
			\sum\limits_{\tau\in[\sigma]}\hat{\phi}(\tau)=\sum\limits_{\tau\in X_\sigma}\hat{\phi}(\tau)=d_\phi^H(S_\sigma) = d_\psi^K(S_\sigma)=\sum\limits_{\tau\in X_\sigma}\hat{\psi}(\tau)=\sum\limits_{\tau\in[\sigma]}\hat{\psi}(\tau).
			\end{equation*} 
	Assume that $\sigma$ is type-II. 
		By Proposition \ref{prop1}, the decomposition of $\sigma$ contains at least one type-II cycle. 
	To show that $\sum\limits_{\tau\in[\sigma]}\hat{\phi}(\tau)=\sum\limits_{\tau\in[\sigma]}\hat{\psi}(\tau)$, for each $\sigma\in S_n$, we will use a strong induction on the number of type-II cycle contained in each permutation. 
		Firstly, suppose that $\sigma$ contains only one type-II cycle, without loss of generality, says $\sigma_1 = (a_1 \, \dots \, a_s)$. 
	By applying (\ref{eq01}), we have
			\begin{equation*}
			X_\sigma = [\sigma] \cup S(\sigma),
			\end{equation*}
	where $S(\sigma) = S(\sigma_1)[\sigma_2][\sigma_3]\cdots[\sigma_k]$. 
		By using Theorem \ref{lem1}, we obtain that
			\begin{equation*}
			\sum\limits_{\tau\in X_\sigma}\hat{\phi}(\tau)=d_\phi^H(S_\sigma) = d_\psi^K(S_\sigma)=\sum\limits_{\tau\in X_\sigma}\hat{\psi}(\tau).
			\end{equation*}
	So, by Proposition \ref{prop4}, we can conclude that $[\sigma] \cap S(\sigma) = \emptyset$, and thus,
			\begin{equation*}
			\sum\limits_{\tau\in [\sigma]}\hat{\phi}(\tau)+\sum\limits_{\tau\in S(\sigma)}\hat{\phi}(\tau)=\sum\limits_{\tau\in X_\sigma}\hat{\phi}(\tau)=\sum\limits_{\tau\in X_\sigma}\hat{\psi}(\tau)=\sum\limits_{\tau\in [\sigma]}\hat{\psi}(\tau)+\sum\limits_{\tau\in S(\sigma)}\hat{\psi}(\tau). 
			\end{equation*}
		Since each element in $S(\sigma_1)$ is type-I, every permutation in $S(\sigma)$ is also type-I. 
	Moreover, by Theorem \ref{lem1}, $d_\phi^H(S_\tau) = \hat{\phi}(\tau)$ and $d_\psi^K(S_\tau) = \hat{\psi}(\tau)$, for each $\tau \in S(\sigma)$. 
		As we have already proven that this theorem holds for every type-I permutation, we have
			\begin{equation*}
			\sum\limits_{\tau\in S(\sigma)}\hat{\phi}(\tau) = \sum\limits_{\tau\in S(\sigma)}d_\phi^H(S_\tau) = \sum\limits_{\tau\in S(\sigma)}d_\psi^K(S_\tau) = \sum\limits_{\tau\in S(\sigma)}\hat{\psi}(\tau). 
			\end{equation*}
	This implies that
			\begin{equation*}
			\sum\limits_{\tau\in [\sigma]}\hat{\phi}(\tau)=\sum\limits_{\tau\in [\sigma]}\hat{\psi}(\tau). 
			\end{equation*}
		Because $\sigma$ is arbitrary, the statement of this theorem is true for every permutation containing exactly one type-II cycle in its decomposition. 
	Suppose the hypothesis induction, that is, the decomposition of $\sigma$ contains exactly $l$ type-II cycles and the statement holds for each permutation containing less than $l$ type-II cycles. 
		By (\ref{eq01}),
			\begin{equation*}
			X_\sigma = [\sigma] \cup X,
			\end{equation*}
		where $X = (\bigcup\limits_{j = 1}^l \bigcup\limits_{(S) \in Seq(j)} X_\sigma^S)$. 
	By applying Proposition \ref{prop4} and Theorem \ref{lem1} again, we have $[\sigma] \cap X = \emptyset$ and
			\begin{equation*}
			\sum\limits_{\tau\in [\sigma]}\hat{\phi}(\tau)+\sum\limits_{\tau\in X}\hat{\phi}(\tau)=\sum\limits_{\tau\in X_\sigma}\hat{\phi}(\tau)=\sum\limits_{\tau\in X_\sigma}\hat{\psi}(\tau)=\sum\limits_{\tau\in [\sigma]}\hat{\psi}(\tau)+\sum\limits_{\tau\in X}\hat{\psi}(\tau). 
			\end{equation*}
		By Proposition \ref{prop4}(3), $[\sigma]$ is the set of all permutations in $X_\sigma$ containing $l$ type-II cycles in its decomposition, that is, the decomposition of each permutation in $X$ containing less than $l$ type-II cycles. 
	By Proposition \ref{prop4}, $X$ is the union of the equivalence classes of elements in $X_\sigma \setminus [\sigma]$. 
		Assume that $\sigma_j = (a_1^j \, \dots \, a_{s_j}^j)$ for every $i \in [l]$. 
	For each $(S) = (x_1, \dots,  x_j)$, due to (\ref{eq10}),
			\begin{equation}\label{eq11}
			X_\sigma^S = \bigcup\limits_{(S_{x_1}, \dots, S_{x_j}) \in \mathbb{I}}[\prod\limits_{i \in S}(\sigma_i)_{(S_i)}\prod\limits_{i \not\in S}\sigma_i],
			\end{equation} 
	where $\mathbb{I} = \{a_1^{x_1}, a_{s_{x_1}}^{x_1}\} \times \cdots \times \{a_1^{x_j}, a_{s_{x_j}}^{x_j}\}$. 
		Since $\prod\limits_{i \in S}(\sigma_i)_{(S_i)}\prod\limits_{i \not\in S}\sigma_i \in X$, its decomposition contains less than $l$ type-II cycles.
	Because of the hypothesis of the induction, we have 
			\begin{equation*}
			\sum\limits_{\tau\in [\prod\limits_{i \in S}(\sigma_i)_{(S_i)}\prod\limits_{i \not\in S}\sigma_i]}\hat{\phi}(\tau)=\sum\limits_{\tau\in [\prod\limits_{i \in S}(\sigma_i)_{(S_i)}\prod\limits_{i \not\in S}\sigma_i]}\hat{\psi}(\tau). 
			\end{equation*}
		Thus, by the definition of $X$ and (\ref{eq11}), we obtain that 
			\begin{equation*}
			\sum\limits_{\tau\in X}\hat{\phi}(\tau)=\sum\limits_{\tau\in X}\hat{\psi}(\tau). 
			\end{equation*}
	This implies that
			\begin{equation*}
			\sum\limits_{\tau\in [\sigma]}\hat{\phi}(\tau)=\sum\limits_{\tau\in [\sigma]}\hat{\psi}(\tau). 
			\end{equation*}
		The induction is done, and hence the statement is true for every permutation in $S_n$.
		
	Now, suppose that (3) occurs. 
		Let $A$ be a symmetric matrix, $\sigma=\sigma_1\cdots\sigma_k$ the decomposition of $\sigma$ and $\omega \in [\sigma]$. 
	Then $\omega = \sigma_1^{n_1}\cdots\sigma_k^{n_k}$, where $n_1,\dots,n_k \in \{1,-1\}$. 
		Claim that $\prod\limits_{i=1}^n [A]_{i \, \sigma(i)} = \prod\limits_{i=1}^n [A]_{i \, \omega(i)}$. 
	In order to prove that (1) holds, we consider $\sigma_j$, for each $1\leq j \leq k$. 
		If $n_j = 1$, then $\sigma(i) = \sigma_j(i) = \omega(i)$ for each $i \in \operatorname{Fix}(\sigma_j)^c$, which implies that
			\begin{equation*}
			\prod\limits_{i \in \operatorname{Fix}(\sigma_j)^c}[A]_{i \, \sigma(i)} = \prod\limits_{i \in \operatorname{Fix}(\sigma_j)^c}[A]_{i \, \omega(i)}.
			\end{equation*} 
	Suppose that $n_j = -1$, that is, $\sigma^{-1}(i) = \sigma_j^{-1}(i) = \omega(i)$ for every $i \in \operatorname{Fix}(\sigma_j)^c$.
		 We obtain that
			\begin{eqnarray*}
			\prod\limits_{i \in \operatorname{Fix}(\sigma_j)^c}[A]_{i \, \sigma(i)}
			& = & \prod\limits_{i \in \operatorname{Fix}(\sigma_j)^c}[A]_{\sigma(i) \, i}\\ 
			& = & \prod\limits_{i \in \operatorname{Fix}(\sigma_j)^c}[A]_{i \, \sigma^{-1}(i)}\\ 
			& = & \prod\limits_{i \in \operatorname{Fix}(\sigma_j)^c}[A]_{i \, \omega(i)}.
			\end{eqnarray*}
	By Lemma \ref{lem0}(3),
			\begin{equation*}
			[n] = \operatorname{Fix}(\sigma) \cup \operatornamewithlimits{Fix}(\sigma_1)^c \cup \operatornamewithlimits{Fix}(\sigma_2)^c \cup \cdots \cup \operatornamewithlimits{Fix}(\sigma_k)^c.
			\end{equation*}
		By recognizing that $\operatorname{Fix}(\sigma) = \operatorname{Fix}(\omega)$, we have
			\begin{eqnarray*}
			\prod\limits_{i=1}^n [A]_{i \, \sigma(i)} 
			&=& \prod\limits_{j = 1}^k\prod\limits_{i \in \operatorname{Fix}(\sigma_j)^c}[A]_{i \, \sigma(i)}\prod\limits_{i \in \operatorname{Fix}(\sigma)}[A]_{i \, \sigma(i)}\\
			&=& \prod\limits_{j = 1}^k\prod\limits_{i \in  \operatorname{Fix}(\sigma_j)^c}[A]_{i \, \omega(i)}\prod\limits_{i \in \operatorname{Fix}(\sigma)}[A]_{i \, \omega(i)}\\ 
			&=& \prod\limits_{i=1}^n [A]_{i \, \omega(i)}.
			\end{eqnarray*} 
	The claim is true. 
		Now, let $[\tau_1],\dots,[\tau_s]$ be all distinct equivalence classes of $S_n$. 
	Then, by the fact above, we can conclude that
			\begin{eqnarray*}
			d_\phi^H(A)
			& = & \sum\limits_{\sigma \in S_n} \hat{\phi}(\sigma) \prod\limits_{i = 1}^n [A]_{i \, \sigma(i)}\\
			& = & \sum\limits_{j = 1}^k\sum\limits_{\sigma \in [\tau_j]} \hat{\phi}(\sigma) \prod\limits_{i = 1}^n [A]_{i \, \sigma(i)}\\
			& = & \sum\limits_{j = 1}^k\sum\limits_{\sigma \in [\tau_j]} \hat{\phi}(\sigma) \prod\limits_{i = 1}^n [A]_{i \, \tau_j(i)}\\
			& = & \sum\limits_{j = 1}^k(\sum\limits_{\sigma \in [\tau_j]} \hat{\phi}(\sigma)) \prod\limits_{i = 1}^n [A]_{i \, \tau_j(i)}\\
			& = & \sum\limits_{j = 1}^k(\sum\limits_{\sigma \in [\tau_j]} \hat{\psi}(\sigma)) \prod\limits_{i = 1}^n [A]_{i \, \tau_j(i)}\\
			& = & \sum\limits_{j = 1}^k\sum\limits_{\sigma \in [\tau_j]} \hat{\psi}(\sigma) \prod\limits_{i = 1}^n [A]_{i \,\tau_j(i)}\\
			& = & \sum\limits_{j = 1}^k\sum\limits_{\sigma \in [\tau_j]} \hat{\psi}(\sigma) \prod\limits_{i = 1}^n [A]_{i \, \sigma(i)}\\
			& = & \sum\limits_{\sigma \in S_n} \hat{\psi}(\sigma) \prod\limits_{i = 1}^n [A]_{i \, \sigma(i)}\\
			& = & d_\psi^K(A).		
			\end{eqnarray*}
		This proves (1). 
	\end{proof}
	
	\begin{cor}
	Let $H$ and $K$ be subgroups of $S_n$ and $\phi$ and $\psi$ complex valued functions of $H$ and $K$, respectively. 
		Suppose that $Y$ is a subset of $\mathbb{S}_n(\mathbb{C})$ containing $S_\sigma$ for every $\sigma \in S_n$. 
	Then $d_\phi^H(A) = d_\psi^K(A)$ for each $A \in Y$ if and only if $d_\phi^H(S_\sigma) = d_\psi^K(S_\sigma)$ for each $\sigma \in S_n$.
	\end{cor}

	By Theorem 2.2. in \cite{2014jm2}, any two generalized matrix functions $d_\phi^H$ and $d_\psi^K$ are equal on the set of all nonsingular matrices if and only if $\hat{\phi} = \hat{\psi}$. 
		Next corollary is obtained immediately by using this fact.
	
	\begin{cor}
	Let $H$ and $K$ be subgroups of $S_n$. 
		Suppose that $\phi$ and $\psi$ be complex valued functions of $H$ and $K$, respectively. 
	If $d_\phi^H(A) = d_\psi^K(A)$ for every nonsingular matrix $A$, then $d_\phi^H(A) = d_\psi^K(A)$ for every $A\in \mathbb{S}_n(\mathbb{C})$.
	\end{cor}
	
	We set $H = S_n = K$. 
		By applying Theorem \ref{thm1}, the results when $\phi$ and $\psi$ are characters are verified.
	
	\begin{cor}
	Let $\phi$ and $\psi$ be characters of $S_n$. 
		Then $d_\phi^{S_n}(A) = d_\psi^{S_n}(A)$ for every $A \in \mathbb{S}_n(\mathbb{C})$ if and only if $\phi = \psi$.
	\end{cor}
	\begin{proof}
	Consider the equivalence class in $S_n$. We can see that each equivalence class is a subset of a conjugacy class. 
		Since any character is a class function, for each $\sigma\in S_n$, we have $\phi(\tau) = \phi(\sigma)$ and $\psi(\tau) = \psi(\sigma)$ for every $\tau \in [\sigma]$. 
	By Theorem \ref{thm1}, this corollary holds.
	\end{proof}
	
	\begin{thm} 
	Let $G \leq S_n$ and $\chi$ be a complex valued function of $G$. 
		Then the following are equivalent.
			\begin{enumerate}
			\item[(1)] $d_\chi^G(A) = d_\chi^G(A^T)$ for each $A \in M_n(\mathbb{C})$,
			\item [(2)] $d_\chi^G(AB) = d_\chi^G(BA)$ for each $A, B \in \mathbb{S}_n(\mathbb{C})$,
			\item[(3)] $\chi(\sigma) = \chi(\sigma^{-1})$ for each $\sigma \in S_n$.
			\end{enumerate}
	\end{thm}
	\begin{proof}
	Since $(AB)^T = BA$ for every symmetric matrices $A$ and $B$, (1) implies (2). 
		Suppose that (2) is true. Let $\sigma \in S_n$. 
	There exist symmetric matrices $A$ and $B$ such that $P_\sigma = AB$, where $P_\sigma$ is the permutation matrix corresponding to $\sigma$. 
		So $P_{\sigma^{-1}} = P_\sigma^T = BA$. 
	By the assertion,
			\begin{equation*}
			\chi(\sigma) = d_\chi^G(P_\sigma) = d_\chi^G(P_{\sigma^{-1}}) = \chi(\sigma^{-1}).
			\end{equation*}
		
	Suppose that (3) holds, that is, $\chi(\sigma) = \chi(\sigma^{-1})$ for every $\sigma \in S_n$. Let $A \in M_n(\mathbb{C})$. 
		Then 
			\begin{eqnarray*}
			d_\chi^G(A) 
			& = &\sum\limits_{\sigma \in G}\chi(
			\sigma)\prod\limits_{i=1}^n [A]_{i \, \sigma(i)}\\
			& = &\sum\limits_{\sigma \in G}\chi(
			\sigma)\prod\limits_{i=1}^n [A]_{\sigma^{-1} \, (i) i}\\
			& = &\sum\limits_{\sigma \in G}\chi(
			\sigma^{-1})\prod\limits_{i=1}^n [A]_{\sigma(i) \, i}\\
			& = &\sum\limits_{\sigma \in G}\chi(
			\sigma)\prod\limits_{i=1}^n [A]_{\sigma(i) \, i}\\
			& = & d_\chi^G(A^T),
			\end{eqnarray*}
		which completes the proof.
	\end{proof}

	Since a character is a class(conjugacy class) function, next corollaries are obtained immediately.
	
	\begin{cor}
	Let $G \leq S_n$ and $\chi$ be a character of $G$. 
		Then the following are equivalent.
			\begin{enumerate}
			\item[(1)] $d_\chi^G(A) = d_\chi^G(A^T)$ for each $A \in M_n(\mathbb{C})$,
			\item [(2)] $d_\chi^G(AB) = d_\chi^G(BA)$ for each $A, B \in \mathbb{S}_n(\mathbb{C})$,
			\item[(3)] $\chi$ is a real valued function.
			\end{enumerate}
	\end{cor}
	
	\begin{cor}
	Let $\chi$ be a character of $S_n$. 
		Then $d_\chi^{S_n}(AB) = d_\chi^{S_n}(BA)$ for every $A,B \in \mathbb{S}_n(\mathbb{C})$.
	\end{cor}
	
	For each $\sigma \in S_n$, define $C_\sigma : [n] \times [n] \rightarrow \{0, 1\}$ by 
			\begin{equation*}
			C_\sigma(i,j) = \left\{
			\begin{array}{ll}
			1, & \hbox{if $i,j \in \operatorname{Fix}(\sigma)^c$,} \\
			0, & \hbox{otherwise.}
			\end{array}
			\right.
			\end{equation*}
	
	\begin{lem}\label{lem3} 
	Let $\sigma = (a_1 \, a_2 \, a_3) \in S_n$. 
		Then
			\begin{equation*}
			[S_\sigma^2]_{i \, j} = C_\sigma(i, j) + \delta_{ij},
			\end{equation*}
		where $\delta_{ij}$ is the Kronecker delta function.
	\end{lem}
	\begin{proof}
	By the definition of $S_\sigma$, we have $[S_\sigma]_{i \, j} \in \{0, 1\}$ and $[S_\sigma]_{i \, j} = 1$ if and only if $j = \sigma(i)$ or $j = \sigma^{-1}(i)$ for each $i, j \in [n]$. 
		Note that
			\begin{equation}\label{eq06}
			[S_\sigma^2]_{i \, j} = \sum\limits_{k=1}^n[S_\sigma]_{i \, k}[S_\sigma]_{k \, j}.
			\end{equation}
	We consider $[S_\sigma]_{i \, j}$, where $i, j \in [n]$, as six possible cases.
		By applying the equation $(\ref{eq06})$ in each case, we can verify this lemma as below. 
		
	Case 1. $i,j \in \operatorname{Fix}(\sigma)$ and $i = j$. 
		Then $C_\sigma(i, j) = 0$ and hence
			\begin{equation*}
			[S_\sigma^2]_{i \, j} = [S_\sigma^2]_{i \, i} = ([S_\sigma]_{i \, i})^2 = 1 = C_\sigma(i, j) + \delta_{ij}.
			\end{equation*}
		
	Case 2. $i,j \in \operatorname{Fix}(\sigma)$ but $i \neq j$. 
		So, we obtain that $[S_\sigma]_{i \, k} = 0$ for each $k \in \{1 \dots, n\} \setminus \{i\}$, which implies that $[S_\sigma]_{i \, j} = 0$. 
	Thus
			\begin{equation*}
			[S_\sigma^2]_{i \, j} = [S_\sigma]_{i \, i}[S_\sigma]_{i \, j} = 0 = C_\sigma(i, j) + \delta_{ij}.
			\end{equation*}
		
	Case 3. $i, j \in \operatorname{Fix}(\sigma)^c$ and $i = j$. 
		Then $[S_\sigma]_{i \, k} = 1$ as long as $k$ is $\sigma(i)$ or $\sigma^{-1}(i)$. 
	This indicates that
			\begin{eqnarray*}
			[S_\sigma^2]_{i\,j} 
			& = & [S_\sigma]_{i\, \sigma(i)}[S_\sigma]_{\sigma(i)\,j} + [S_\sigma]_{i\, \sigma^{-1}(i)}[S_\sigma]_{\sigma^{-1}(i)\,j}\\
			& = & [S_\sigma]_{i\, \sigma(i)}[S_\sigma]_{\sigma(i)\,i} + [S_\sigma]_{i\, \sigma^{-1}(i)}[S_\sigma]_{\sigma^{-1}(i)\,i}\\
			& = & 2\\
			& = & C_\sigma(i, j) + \delta_{ij}.
			\end{eqnarray*}
		
	Case 4. $i, j \in \operatorname{Fix}(\sigma)^c$ but $i \neq j$. 
		By a similar argument as above,
			\begin{equation*}
			[S_\sigma^2]_{i\,j} = [S_\sigma]_{i\, \sigma(i)}[S_\sigma]_{\sigma(i)\,j} + [S_\sigma]_{i\, \sigma^{-1}(i)}[S_\sigma]_{\sigma^{-1}(i)\,j}.
			\end{equation*}
	Since $|\operatorname{Fix}(\sigma)| = 3$ and $i \neq j$, either $j = \sigma(i)$ or $j = \sigma^{-1}(i)$, which implies
			\begin{equation*}
			[S_\sigma^2]_{i\,j} = 1 = C_\sigma(i, j) + \delta_{ij}.
			\end{equation*}
		
	Case 5. $i \in \operatorname{Fix}(\sigma)$ and $j \in \operatorname{Fix}(\sigma)^c$. 
		Then $j \neq \sigma(i)$, that is, $[S_\sigma]_{i\,j} = 0$. 
	Thus
			\begin{equation*}
			[S_\sigma^2]_{i\,j} = [S_\sigma]_{i\,i}[S_\sigma]_{i\,j} = 0 = C_\sigma(i, j) + \delta_{ij}.
			\end{equation*}
		
	Case 6. $i \in \operatorname{Fix}(\sigma)^c$ and $j \in \operatorname{Fix}(\sigma)$. 
		By using a similar reasoning as Case 5., 
			\begin{equation*}
			[S_\sigma^2]_{i\,j} = [S_\sigma]_{i\,j}[S_\sigma]_{j\,j} = 0 = C_\sigma(i, j) + \delta_{ij}.
			\end{equation*}
	Thus the proof is completed.
	\end{proof}
	
	Denote $F_3^c(n) := \{\sigma \in S_n \mid |\operatorname{Fix}(\sigma)^c| \leq 3\}$. 
		By using Lemma \ref{lem3}, the following theorem is obtained.
	\begin{thm}\label{thm15}
	Let $G$ be a subgroup of $S_n$ and $\chi$ a character of $G$. 
		Then $d_\chi^G = \det$ if and only if $d_\chi^G(S_\sigma)d_\chi^G(S_\tau) = d_\chi^G(S_\sigma S_\tau)$, for every $\sigma, \tau \in F_3^c(n)$.
	\end{thm}
	\begin{proof}
	Let $\sigma, \tau \in F^c_3(n)$. 
		Clearly, $d_\chi^G(S_\sigma)d_\chi^G(S_\tau) = d_\chi^G(S_\sigma S_\tau)$ if $d_\chi^G = \det$. 
	Suppose that the converse hypothesis is true. 
		Since $S_{id} = I$ and $\chi(id) = \det(S_{id})$, by the assumption, we have
			\begin{equation*}
			\chi(id) = d_\chi^G(S_{id}) = (d_\chi^G(S_{id}))^2 = (\chi(id))^2.
			\end{equation*}
	Because $\chi$ is a character, $\chi(id) \neq 0$, that is, $\chi(id) = 1$. 
		Thus $\chi$ is linear. 
	If $n = 1$, then $d_\chi^G(A) = [A]_{11} = \det(A)$ for arbitrary $A \in M_n(\mathbb{C})$. 
		For each transposition $\omega$ in $S_n$, we have $\omega^2 = id$, so, by the assumption,
			\begin{equation*}
			(\hat{\chi}(\omega))^2 = (d_\chi^G(S_\omega))^2 = d_\chi^G(I) = 1.
			\end{equation*}
	This implies that $\hat{\chi}(\omega)$ is either $1$ or $-1$, that is, $\omega \in G$. 
		If $n = 2$ and $\chi(1 \, 2) = 1$, then
			\begin{equation*}
			128 =
			d_ \chi^G\left(\begin{bmatrix}
			8 & 8 \\
			8 & 8 
			\end{bmatrix}\right)=
			d_ \chi^G\left(\begin{bmatrix}
			2 & 2 \\
			2 & 2 
			\end{bmatrix}\right)
			d_ \chi^G\left(\begin{bmatrix}
			2 & 2 \\
			2 & 2 
			\end{bmatrix}\right) = 64,
			\end{equation*}
		which is a contradiction. 
	This implies that $\chi(1 \, 2) = -1$, and thus $d_\chi^G = \det$. 
		Suppose that $n \geq 3$. 
	Since the set of all transpositions is the generating set of $S_n$, $G = S_n$. 
		For each $\sigma \in S_n$, we can write
			\begin{equation*}
			\sigma = \theta_1 \theta_2 \cdots \theta_k,
			\end{equation*}
		where $\theta_m$ is a transposition for each $m \in [k]$. 
	Because $\chi$ is linear, we can conclude that
			\begin{equation*}
			\chi(\sigma) = \chi(\theta_1)\chi(\theta_2) \cdots \chi(\theta_k).
			\end{equation*}
		Thus, to verify that $d_\chi^G = \det$, it suffices to show that $\chi(\omega) = -1$ for each transposition $\omega$. 
	Let $\sigma = (a_1 \, a_2 \, a_3)$. 
		By Theorem \ref{lem1} and the assumption,
			\begin{equation}\label{eq03}
			d_\chi^G(S_\sigma^2) = (d_\chi^G(S_\sigma))^2 = (\chi(a_1 \, a_2 \, a_3) + \chi(a_1 \, a_3 \, a_2))^2.
			\end{equation}
	By Lemma \ref{lem3}, we have $[S_\sigma^2]_{i j} = 0$ if and only if $i \neq j$ and at least one of $i, j$ is in $\operatorname{Fix}(\sigma)$ and the other is in $\operatorname{Fix}(\sigma)^c$. 
		Thus, for each $\tau \in B := \{id, (a_1 \, a_2), (a_1 \, a_3), (a_2 \, a_3), \sigma, \sigma^{-1}\}$, $\prod\limits_{i = 1}^n[S_\sigma^2]_{i \, \tau(i)} \neq 0$ while $\prod\limits_{i = 1}^n[S_\sigma^2]_{i \, \pi(i)} = 0$ for every $\pi \in S_n \setminus B$. 
	Therefore,
			\begin{eqnarray*}
			d_\chi^G(S_\sigma^2) 
			& = & \sum\limits_{\tau \in S_n}\chi(\tau)\prod\limits_{i = 1}^n[S_\sigma^2]_{i \, \tau(i)}\\ 
			& = & \sum\limits_{\tau \in B}\chi(\tau)\prod\limits_{i = 1}^n (C_\sigma(i, \tau(i)) + \delta_{i \, \tau(i)})\\
			& = & \chi(\sigma^{-1}) + \chi(\sigma) + 2\chi(a_1 \, a_2) + 2\chi(a_1 \, a_3) 
			+ 2\chi(a_2 \, a_3) + 8\chi(id).
			\end{eqnarray*}
		Note that
			\begin{equation*}
			\sigma = (a_1 \, a_3)(a_1 \, a_2) = (a_1 \, a_2)(a_2 \, a_3) = (a_2 \, a_3)(a_1 \, a_3)
			\end{equation*}
		and
			\begin{equation*}
			\sigma^{-1} = (a_1 \, a_2)(a_1 \, a_3) = (a_2 \, a_3)(a_1 \, a_2) = (a_1 \, a_3)(a_2 \, a_3).
			\end{equation*}
	Since $\chi$ is linear, we can compute that
			\begin{equation}\label{eq07}
			\hbox{$\chi(a_1 \, a_2) = \chi(a_1 \, a_3) = \chi(a_2 \, a_3)$ and $\chi(\sigma) = (\chi(a_1 \, a_2))^2 = \chi(\sigma^{-1})$.}
			\end{equation}
		By applying $(\ref{eq07})$, we have $\chi(\sigma) = 1$, which implies that
			\begin{equation}\label{eq04}
			d_\chi^G(S_\sigma^2) = 6\chi(a_1 \, a_2) + 10.
			\end{equation}
	Similarly, by applying $(\ref{eq07})$ to (\ref{eq03}), we obtain that
			\begin{equation}\label{eq05}
			d_\chi^G(S_\sigma^2) = 4.
			\end{equation}
		Consider (\ref{eq04}) and (\ref{eq05}), we have $\chi(a_1 \, a_2) = -1$. 
	Since $a_1, a_2$ are arbitrary, $\chi(\omega) = -1$ for every transposition $\omega$ in $S_n$. 
		Hence $d_\chi^G = \det$.
	\end{proof}
	
	Due to Theorem \ref{thm15}, when we define a generalized matrix function by using a character of a subgroup of $S_n$, $\det$ is the unique generalized matrix function preserving the product on every subset of $M_n(\mathbb{C})$ containing $F_3^c(n)$.
		An obvious example of such subset is $\mathbb{S}_n(\mathbb{C})$.
	Moreover, by Theorem \ref{lem1}, $S_\sigma$ is nonsingular for every $\sigma \in F_3^c(n)$.
		So, $GL_n(\mathbb{C})$ is also a subset of $M_n(\mathbb{C})$ containing $F_3^c(n)$.
	
	\section*{Acknowledgment} The second author would like to thank Faculty of Sciences, Naresuan University, for the financial support on the project number R2561E003.
	\bibliography{mybibfile1}
	
\end{document}